\documentclass{conm-p-l}

\usepackage{amssymb,amscd,amsmath}

\newtheorem{theorem}{Theorem}[section]
\newtheorem{lemma}[theorem]{Lemma}
\newtheorem{corollary}[theorem]{Corollary}
\newtheorem{proposition}[theorem]{Proposition}

\theoremstyle{definition}
\newtheorem{definition}[theorem]{Definition}
\newtheorem{example}[theorem]{Example}

\newtheorem{conjecture}[theorem]{Conjecture}

\theoremstyle{remark}
\newtheorem{remark}[theorem]{Remark}

\numberwithin{equation}{section}

\newcommand{\BbR}{{\mathbb R}}

\newcommand{\BbC}{{\mathbb C}}
\newcommand{\BbZ}{{\mathbb Z}}
\newcommand{\pdo}{\Psi{\rm DO}}
\newcommand{\bpdo}{\overline{\Psi{\rm DO}_0^*}}
\newcommand{\calg}{{\mathfrak g}}
\newcommand{\gl}{{\mathfrak gl}}
\newcommand{\GG}{\pdo_0^*}
\newcommand{\cklo}{c_k^{\rm lo}}
\newcommand{\ckw}{c_k^{\rm w}}
\newcommand{\tr}{{\rm tr}}
\newcommand{\bpsi}{\bar\Psi}

\newcommand{\pdoo}{\pdo_{\leq 0}}
\newcommand{\pdomo}{\pdo_{ -1}}
\newcommand{\pdomi}{\pdo_{-\infty}}
\newcommand{\pdom}{\pdo_{-\infty}}

\newcommand{\mapsnm}{{\rm Maps}^{s_0}(N,M)}
\newcommand{\mapsnmone}{{\rm Maps}(N,M)}
\newcommand{\mapsfnm}{{\rm Maps}_f(N,M)}
\newcommand{\el}{{\rm Ell}^*}
\newcommand{\gli}{{\rm GL}_\infty}
\newcommand{\opa}{{\rm Op}_1(A)}
\newcommand{\op}{{\rm Op}_1}
\newcommand{\bop}{\overline{\rm Op}_1}
\newcommand{\glh}{\gl(\mathcal H)}

\newcommand{\dvol}{{\rm dvol}}

\newcommand{\calM}{\mathcal M}
\newcommand{\calC}{\mathcal C}
\newcommand{\calA}{\mathcal A}
\newcommand{\calF}{\mathcal F}
\newcommand{\calH}{\mathcal H}
\newcommand{\calE}{\mathcal E}
\newcommand{\calG}{\mathcal G}
\newcommand{\ev}{{\rm ev}}
\newcommand{\resw}{{\rm res}^{\rm w}}

\begin{document}

\title[Characteristic Classes and $\pdo$s]{Characteristic Classes and  Zeroth Order Pseudodifferential Operators}


\author[A. Larrain-Hubach]{Andr\'es Larrain-Hubach}
\address{Department of Mathematics and Statistics, Boston University}

\email{alh@bu.edu}

\author[S. Rosenberg]{Steven Rosenberg}
\address{Department of Mathematics and Statistics, Boston University}
\email{sr@math.bu.edu}

\author[S. Scott]{Simon Scott}
\address{Department of Mathematics, King's College}
\email{simon.scott@kcl.ac.uk}

\author[F. Torres-Ardila]{Fabi\'an Torres-Ardila}
\address{Metropolitan College, Boston University}
\email{fatorres@math.bu.edu}


\begin{abstract}  We provide evidence for the conjecture
 that the Wodzicki-Chern classes vanish for all bundles with 
invertible zeroth order pseudodifferential operators $\GG$ as structure group.  The leading order Chern classes are nonzero in general, and they detect elements of the
de Rham cohomology groups $H^*(B\GG, \BbC).$
\end{abstract}

\maketitle

\centerline{\today}
\centerline{To Misha Shubin}

\section{Introduction}

Infinite dimensional manifolds $\calM$ such  as the loop space Maps$(S^1,M)$ of a manifold or more generally the space  Maps$(N,M)$  of maps between 
manifolds have interesting geometry.  The structure group of these infinite dimensional manifolds 
(i.e. of their tangent bundles $\calE$) is a gauge group of a finite rank bundle $E\to N$ over the source space.  When the manifolds have 
Riemannian metrics, the mapping spaces have Levi-Civita connections with connection and 
curvature forms taking values in $\pdo_{\leq 0} = \pdo_{\leq 0}(E)$, the algebra of nonpositive integer order classical pseudodifferential operators ($\pdo$s) acting
 on sections of $E$.    Thus for geometric purposes, the structure group should be enlarged to
 $\GG$, the group of invertible zeroth order classical $\pdo$s, since at least formally Lie$(\GG) = \pdo_{\leq 0}.$

 As discussed in \cite{L-P}, \cite{MRT}, \cite{P-R2}, the generalizations of Chern-Weil theory to $\GG$-bundles 
 are classified by the set of traces on $\pdo_{\leq 0}$, i.e. by the Hochschild cohomology group
 $HH^0(\pdo_{\leq 0}, \BbC).$  Indeed, given such a trace $T:\pdo_{\leq 0}\to \BbC$, one defines $T$-Chern classes of a connection with curvature $\Omega\in\Lambda^2(\calM, \pdo_{\leq 0})$ by the de Rham class
 $c_k^T(\calE) = [T(\Omega^k)]\in H^{2k}(\calM, \BbC).$  
 
 These traces roughly break into two classes:  the Wodzicki residue and the integral of the zeroth/leading  order symbol
 over the unit cosphere bundle.  In \S2, we prove that the Wodzicki-Chern classes vanish if 
 the structure group reduces from $\GG$ to the subgroup with leading order symbol given by the identity.  
We conjecture that the Wodzicki-Chern classes always vanish, and sketch a possible 
superconnection proof.
 These vanishing results, which were previously known only for loop spaces, reinforce the importance of the nontrivial Wodzicki-Chern-Simons classes produced in \cite{MRT}.  
 
 In the beginning of \S3,  we discuss the analytic and topological issues involved with  universal bundle calculations of Chern classes associated to the leading order symbol trace.  
 The main issue is that the classifying space $B\pdo_0^*$ may not be a manifold, so we want to 
 extend the leading order symbol trace to an algebra whose corresponding classifying space is clearly a manifold.   
 \S3 is devoted to the analytic issue of extending the leading order symbol trace to a Lie algebra containing $\pdo_{\leq 0}$.  It is easier to work with the quotient of $\pdo_{\leq 0}$ by smoothing operators, and in Proposition \ref{prop extend} we find an extension  (really a factorization)
 of the leading order symbol trace to a
 relatively large subalgebra ${\mathcal Q}$ inside a quotient 
 of the set of all bounded operators on a Hilbert space.  These results may not be optimal, so this section should be considered work in progress.
 
 Unfortunately, the classifying space $BQ$ associated to $\mathcal Q$ may not be a manifold, so we are unable to construct universal geometric characteristic classes.  In \S4, we take a smaller
 extension of the leading order symbol trace with corresponding classifying space a manifold.  We then
show that the leading order Chern classes of gauge bundles are nontrivial in general.
 This implies that there is a topological theory of characteristic classes of $\GG$-bundles involving 
the cohomology of $B\GG$.  As a first step toward understanding this cohomology,
we use the nonvanishing of leading order Chern classes on mapping spaces
to show in Theorem \ref{last theorem} that for $E^\ell\to N$,
 $H^*(B\GG, \BbC)$ surjects onto 
the polynomial algebra $H^*(BU(\ell), \BbC) = \BbC[c_1(EU(\ell)),\ldots,c_\ell(EU(\ell))]$.
This complements Rochon's work \cite{R} on the 
 homotopy groups of a certain stablilization of $\GG$.  The proof shows that $H^*(B\calG,\BbC)$ also surjects onto $H^*(BU(\ell),\BbC)$, where $\calG$ is the gauge group of $E$. 
 For comparison,
 $H^*(B\calG_0,\BbC)$, where $\calG_0$ is the group of based gauge transformations, has been completely determined by different methods, and $H^*(B\calG,\BbC)$ is known if the center of the underlying finite dimensional Lie group is finite
\cite[p.~181]{dk}.

As much as possible, we sidestep the trickier analytic and topological aspects of $B\GG$
by working with
de Rham cohomology only.  The questions of whether $\GG$ is a tame Fr\'echet
space \cite{H}, \cite{Omori} and so has a good exponential map, the relationships among
the cohomology with respect to the norm topology, the Fr\'echet topology and intermediate Banach norm topologies, and whether the de Rham 
theorem holds for $B\GG$  \cite{beggs} are not addressed.
 
 The role of $\GG$ in infinite dimensional geometry was explained to us by Sylvie Paycha, and we gratefully acknowledge our many conversations with her.  We also would like to thank Varghese 
 Mathai for suggesting we consider the closure of $\GG$ discussed below.   
The referee both pointed out serious errors and gave valuable 
suggestions for 
simplifying and clarifying this paper, which we gratefully acknowledge.  

Finally,  this paper is in many ways inspired by the seminal text \cite{shubin} of Misha Shubin, whose clear writing has made a difficult subject accessible to so many mathematicians.

\section{Vanishing of Wodzicki-Chern classes of $\pdo$-bundles}

Let $\GG = \GG(E)$ be the group of  zeroth order invertible classical $\pdo$s acting on sections of  a fixed finite rank complex bundle $E\to N.$
We recall the setup for $\GG$-bundles. Fix a complete Riemannian
metric on $N$ and a Hermitian metric on $E$.  For a real parameter $s_0 \gg 0$, 
let $H^{s_0}\Gamma(E) = H^{s_0}\Gamma(E\to N)$ be the sections of $E$ of Sobolev class $s_0.$  This space depends on the
Riemannian metric if $N$ is noncompact, and of course can be defined via local charts without choosing a metric.  Let $\calE$ be a Banach bundle over a base $B$ such that locally $\calE|_U \simeq U\times H^{s_0}\Gamma(E)$ and such that the transition 
functions lie in $\GG(E).$  Then we call $\calE$ a $\GG$- or $\GG(E)$-bundle over $B$.

The role of $s_0$ is not very important.  We could take the $C^\infty$ Fr\'echet topology on the sections
of $E$, since this is a tame Fr\'echet space in the sense of \cite{H}.

As explained in \cite{eells}, the tangent bundle to  Maps${}^{s_0}(N,M)$, the space of $H^{s_0}$ maps between manifolds $N, M$, is a 
 $\GG$-bundles.  
Fix a component Maps${}^{s_0}_f(N,M)$ of a map $f:M\to N$.  We can take $H^{s_0}\Gamma(f^*TM \to N)$ as the tangent space $T_f{\rm Maps}^{s_0}(N,M) = 
T_f{\rm Maps}_f^{s_0}(M,N).$  Exponentiating sufficiently short 
sections \\
$X\in H^{s_0}\Gamma(f^*TM)$ via $n\mapsto \exp_{f(n),M} X_n$ gives a coordinate neighborhood
of $f$ in\\
 Maps${}^{s_0}(N,M)$, making Maps${}^{s_0}(N,M)$ into a Banach manifold.  The transition maps for
$T{\rm Maps}^{s_0}(N,M)$
 for nearby maps are of the form $d\exp_{f_1}\circ d\exp_{f}^{-1}$, 
which are easily seen to be isomorphic to gauge transformations of $f^*TM.$  Since gauge 
transformations are invertible multiplication operators, $T{\rm Maps}^{s_0}(N,M)$ is a $\GG$-bundle, although at this point there is no need to pass from gauge bundles to $\GG$-bundles.
  Note that the gauge group depends on the component of $f$. 
In particular, for the loop space $LM = {\rm Maps}^{s_0}(S^1,M)$, each complexified tangent space $T_\gamma LM$ is $H^{s_0}\Gamma(S^1\times \BbC^n\to S^1)$ for $M^n$ oriented.  For convenience, we 
will always complexify tangent bundles.

\begin{remark} These bundles fit into the framework of the families index theorem. Start with  a fibration of manifolds with an auxiliary  bundle
$$\begin{CD}  @. E\\
@.@VVV\\ 
Z@>>> M\\
@.@VV{\pi}V\\
@. B
\end{CD}
$$
Here $M$ and $B$ are manifolds, and the fiber is modeled on a compact manifold $Z$.  The 
structure group of the fibration is Diff$(Z)$.  
Pushing down the sheaf of sections of $E$ via $\pi$ 
gives an infinite rank bundle
$$\begin{CD} H^{s_0}\Gamma(E_b) @>>> \calE = \pi_*E\\
@.@VVV\\
@. B
\end{CD}
$$
with the fiber modeled on the $H^{s_0}$ sections of $E_b = E|_{Z_b}$ over $Z_b = \pi^{-1}(b)$ for 
one $b$ in each component of $B$.   The structure 
group of $\calE$ is now a semidirect product $\calG\ltimes {\rm Diff}(Z)$, where $\calG$ is the gauge group of $E_b.$
In particular, for $\ev:N\times {\rm Maps}^{s_0}(N,M) \to M, \ev(n,f) = f(n)$, and
$$\begin{CD} @. E = \ev^*TM @>>> TM\\
@. @VVV @VVV\\
N@>>> N\times {\rm Maps}^{s_0}(N,M) @>{\ev}>> M\\
@.@VV{\pi = \pi_2}V@.\\
@. {\rm Maps}^{s_0}(N,M) @.
\end{CD}
$$
we get $\calE = \pi_*\ev^*TM = T{\rm Maps}^{s_0}(N,M).$  Since the fibration is trivial, the structure group is just the gauge group.  Defining characteristic classes for nontrivial fibrations is open at present.
\end{remark}

As explained in the introduction, any trace on the gauge group $\calG$ of $E\to N$ will give a Chern-Weil
theory of characteristic classes on $\calE.$  Such a trace will be used in \S3.  However, 
using a wider class of traces is natural, as we now explain. The
choice of Riemannian metrics on $N, M$ leads to a family of Riemannian metrics on Maps$(N,M).$  Namely, pick $s \gg 0, s \leq s_0$.  For $X, Y \in T_f{\rm Maps}^{s_0}(N,M)$, set
\begin{equation}\label{one}
\langle X, Y\rangle_{f,s} = \int_N \langle X_n, (1+\Delta)^{s}Y_n\rangle_f(n)\ \dvol(n),
\end{equation}
where $\Delta = (\ev^*\nabla)^*(\ev^*\nabla)$ and $\nabla = \nabla^M$ is the Levi-Civita connection on $M$. Here we assume $N$ is compact.  For example, when $N = S^1$, 
$\ev^*\nabla$ is covariant differentiation along the loop $f$.  Equivalently, we are taking the 
$L^2$ inner product of the pointwise $H^{s}$ norms of $X$ and $Y$.  

The metric (\ref{one}) gives rise to a Levi-Civita connection $\nabla^{s}$ by the Koszul
formula
\begin{eqnarray}\label{two} 
\langle\nabla^s_YX,Z\rangle_s &=&\ X\langle Y,Z\rangle_s +Y\langle X,Z\rangle_s
-Z\langle X,Y\rangle_s\\
&&\qquad 
+\langle [X,Y],Z\rangle_s+\langle [Z,X],Y\rangle_s -\langle[Y,Z],X\rangle_s.\nonumber
\end{eqnarray}
provided the right hand side is a continuous linear functional of $Z\in$\\
$T_f{\rm Maps}^{s_0}(N,M).$  As explained in \cite{MRT}, the only problematic
term $Z\langle X,Y\rangle_s$ is 
continuous in $Z$ for $s\in \BbZ^+$, but this probably fails otherwise.  

Restricting ourselves to $s\in \BbZ^+$, we find  that the
connection one-form and curvature two-form of $\nabla^s$ take values in $\pdo_{\leq 0}(\ev^*TM).$ (This is \cite[Thm. 2.1, Prop. 2.3]{MRT} for $LM$, and the proof generalizes.)  Because these natural
connections do not take values in $\Gamma {\rm End}(E,E)$, the Lie algebra of the gauge group Aut$(E)$,
we have to extend the structure group of $T{\rm Maps}^{s_0}(N,M)$ to $\GG.$  Note that $\GG$ acts as bounded operators on $T_f{\rm Maps}^{s_0}(N,M)$ for all choices of $s_0$,
so the structure group is independent of this choice.
The
zeroth order parts of the connection and curvature forms
are just the connection and curvature forms of $\ev^*\nabla^M_{f(n)}$, 
so only the negative order parts contain new information.  

To extract the new information, we pick the unique trace on $\pdo_{\leq 0}(E\to N)$ that detects negative order terms, namely the Wodzicki residue 
\begin{equation}\label{three} \resw:\pdo_{\leq 0}\to\BbC,\ \resw(A) = (2\pi)^{-n}\int_{S^*N}
 \tr\ \sigma_{-n}(A)(x, \xi)
\ d\xi \dvol(x),
\end{equation}
where $S^*N$ is the unit cosphere bundle of $N^n$.  

We now pass to general $\GG$-bundles $\calE\to \calM$ which admit connections; for example, if $\calM$ admits a partition of
unity, then $\calE\to\calM$ possesses connections.
Standard Chern-Weil theory extends to 
justify the following definition.

\begin{definition}  The k${}^{\rm th}$ Wodzicki-Chern class $\ckw(\calE)$ of the $\GG$-bundle $\calE\to \calM$ admitting a connection $\nabla$ is 
the de Rham cohomology class $[\resw(\Omega^k)]$, where $\Omega$ is the curvature of $\nabla.$
\end{definition}

As for $\mapsnm$, the Wodzicki-Chern classes are independent of the choice of Sobolev parameter $s_0$.
These classes easily vanish for $\GG$-bundles such as $T{\rm Maps}^{s_0}(N,M)$ which restrict to gauge bundles.  
For such bundles admit a connection taking values in the Lie algebra $\Gamma
{\rm End}(E,E)$ of the gauge group.  The curvature two-form is thus a multiplication operator with vanishing
Wodzicki residue.  Since the Wodzicki-Chern class is independent of connection, these classes 
vanish.

We give a subclass of $\GG$-bundles for which the Wodzicki-Chern classes vanish. Recall that 
paracompact Hilbert 
manifolds admit partitions of unity.  Since  $\mapsnm$ is a Hilbert 
manifold and a metrizable
space for $M, N$ closed, it is paracompact and so admits partitions of unity.   Moreover, by a theorem of Milnor, $\mapsnm$
 has the
homotopy type of a  CW complex in the compact-open topology.  This carries over to the Sobolev topology by e.g. putting a connection $\nabla$ 
on $M$, and using the heat operator associated to $\nabla^* \nabla$  to 
homotop  continuous maps to smooth maps.

\begin{theorem}  
 Let $\el\subset \GG$ be the subgroup of invertible zeroth order elliptic operators whose leading order symbol is the identity.  Assume $\calM$ is a manifold homotopy equivalent to a CW complex and admitting a cover with a subordinate partition of unity. If $\calE\to\calM$ is an $\el$-bundle, then the Wodzicki-Chern classes $\ckw(\calE)$ are zero.

\end{theorem}

\begin{proof}
 By \cite[Prop. 15.4]{bw}, for $V = S^n$ or $B^n$, a continuous map $f:V\to \el$  is homotopic
within $\el$ to a map $g:V\to \gli$, the set of operators of the form $I + P$, where $P$ is a 
finite rank operator.
For $V = S^n$, we get that the inclusion $i:\gli\to\el$ induces a surjection 
$i_*:\pi_k(\gli)\to \pi_k(\el)$ on all homotopy groups, and for $V = B^n$ we get that $i_*$ is 
injective.  From the diagram
$$\begin{CD}
 @>>> \pi_k(\gli) @>>> \pi_k(E\gli)@>>> \pi_k(B\gli)@>>>\\
@. @Vi_* VV     @V(Ei)_* VV  @V(Bi)_*VV\\
@>>> \pi_k(\el) @>>> \pi_k(E\el)@>>> \pi_k(B\el)@>>>
\end{CD}$$
we get $(Bi)_*:\pi_k(B\gli)\to \pi_k(B\el)$ is an isomorphism for all $k$. 
These classifying spaces are weakly equivalent to CW complexes \cite[Thm. 7.8.1]{spanier}. 
This implies that $[X, B\gli] = [X,B\el]$ for any CW complex $X$.  In particular, any $\el$-bundle reduces to a $\gli$-bundle.

Thus $\calE\to\calM$ admits a $\gli$-connection.
In fact, the proof of \cite[Prop. 15.4]{bw} implicitly 
shows that $\gli$ is homotopy equivalent to $\Psi_{-\infty}^*$, the group of invertible $\pdo$s of the form $I + P$, where $P$ is a finite rank operator given by a smooth kernel.  Thus we may assume that the connection one-form $\theta$ takes values in
Lie$(\Psi_{-\infty}^*)$, the space of smoothing operators. 
 The curvature two-form is given locally by
$\Omega_\alpha = d\theta_\alpha +\theta_\alpha\wedge \theta_\alpha$, and hence
$\Omega^k$  also takes values in smoothing operators. 
 The Wodzicki residue of $\Omega^k$ therefore vanishes, 
so $\ckw(\calE) = 0.$

\end{proof}

Based on this result and calculations, the following conjecture seems plausible.

\begin{conjecture}  The Wodzicki-Chern classes vanish on any $\GG$-bundle over a base manifold
admitting a partition of unity.
\end{conjecture}

We will see in \S3 that $\el$ is not a deformation retraction of $\GG$ in general, so 
the conjecture does not follow from the previous theorem.

We now outline a putative  proof of the conjecture based on 
the families index theorem setup for 
trivial fibrations as in Remark 2.1; details will appear elsewhere. 
As noted above, the Wodzicki-Chern classes vanish for these gauge bundles, but the
superconnection techniques given below may generalize to other, perhaps all, $\GG$-bundles.

Let  $\nabla = \nabla^0 \oplus \nabla^1$ be a graded connection on a graded infinite dimensional
 bundle $\calE   = \calE^0 \oplus  \calE^1$  over a base space $B$. Let $R   = R^0 \oplus  R^1$ be the corresponding curvature form.  The connection and curvature forms take values in $\pdo$s of
 nonpositive order.

We choose a smooth family of nonpositive order $\pdo$s
$ a : \calE^0 \rightarrow \calE^1$, 
and set
$$ A : \calE \rightarrow \calE, \ A = \left(\begin{array}{cc} 0&a^*\\a&0\end{array}\right).$$
We form the superconnection
$B_t = \nabla + t^{1/2}A.$
For convenience,  assume that $A$ has constant order zero. Then the heat operator
$\exp(-B_t^2)$
is a smooth family of zero order $\pdo$s, as seen from an analysis of its asymptotic expansion.

The standard transgression formula for the local families index theorem is of the form
${\rm str}(\exp(-B_{t_1}^2)) - {\rm str}
(\exp(-B_{t_2}^2)) = d\int_{t_1}^{t_2} \alpha_t dt$,
where str is the supertrace
and $\int\alpha_t$ is an explicit Chern-Simons form \cite{B2}.  In the $t\to\infty$ limit, the connection becomes a connection on the finite rank index bundle, and defines there a smoothing operator on which the Wodzicki residue trace is zero. As $t \to  0$, the Wodzicki residue of $e^{-B_t^2}$ approaches 
the residue of $e^{-R} $.
This limit exists because we are using the Wodzicki residue and not the classical trace, as again follows from an analysis of the symbol asymptotics. This shows that the Wodzicki-Chern character
and hence the $\ckw(\calE)$  vanish 
in cohomology.  

We can manipulate the
choice of $A$ and the bundle $\calE^1$ to infer the result  for a non-graded bundle. That is, we take $\calE^0, \nabla^0, R^0$ to be a fixed $\pdo_0^*$-bundle with connection
and curvature form, and take $\calE^1$ 
to be a trivial bundle with a flat connection $\nabla^1$. 
 Choose $a$ to be an elliptic family of $\pdo$s of order zero parametrized by $B$. 
  Then the graded Wodzicki-Chern character reduces to the Wodzicki-Chern character of $E^0$, and we are done.
 
 It may be that a refined version of this argument gives the vanishing of the Wodzicki-Chern 
 character as a differential form, in which the Wodzicki-Chern-Simons classes of \cite{MRT}
 would always be defined.

\section{Extending the leading order symbol trace}

In contrast to the Wodzicki-Chern classes, the leading order Chern classes are often nonzero, and 
we can use them to detect elements of $H^*(B\GG,\BbC).$  As above, let $\calE$ be a $\GG$-bundle with fiber modeled on $H^{s_0}\Gamma(E\to N).$

Throughout this section, we use the following conventions:  (i) the manifold $N$ is closed and
connected; (ii) all cohomology groups $H^*(X,\BbC)$ are de Rham cohomology; (iii) 
Maps$(N,M)$ denotes $\mapsnm$ for a fixed large Sobolev parameter $s_0$; (iv) smooth
sections of a bundle $F$ are denoted by $\Gamma F.$

\begin{definition}  The k${}^{\rm th}$ leading order Chern class $\cklo(\calE)$ of the $\GG$-bundle $\calE\to \calM$ admitting a connection $\nabla$ is 
the de Rham cohomology class of
$$\int_{S^*N} \tr\ \sigma_0(\Omega^k)(n,\xi)\  d\xi\ \dvol(n),$$
where $\Omega$ is the curvature of $\nabla.$ 
\end{definition}

The point is that the {\it leading order symbol trace} $\int_{S^*N} \tr\ \sigma_0:\pdo_{\leq 0}\to\BbC$ is a trace on this subalgebra, although it does not extend to a trace on all $\pdo$s.

An obvious approach to calculating leading order classes would be to find a universal connection on $E\GG\to B\GG.$  However,
it seems difficult to build a model of $B\GG$ more concrete than the general Milnor construction.  In particular, it is not clear that $B\GG$ is a manifold, so the existence of a connection on $E\GG$ may be moot. 

Alternatively,
since elements of $\GG$ are bounded operators on the Hilbert space $\calH = H^{s_0}\Gamma(E\to N)$, we can 
let $\bar\Psi = \bpdo$ be the closure of $\pdo_0^*$ in $GL(\mathcal H)$ in the norm topology. 
 $\bpsi$ 
acts freely on the contractible space 
$GL(\calH)$, so $E\bpsi = GL(\calH)$ and $B\bpsi = E\bpsi/\bpsi$.
 $GL( \calH)$ is a 
Banach manifold, and since the Frobenius theorem holds in this context, $B\bpsi$ is also a Banach 
manifold \cite{Omori}.  In particular, $E\bpsi\to B\bpsi$ admits a connection.  (It would be interesting to know if $E\bpsi$ 
has a universal connection.)  Unfortunately, it is not clear that the leading order symbol trace extends to ${\rm Lie}(\bar\Psi)$, so defining leading order symbol classes for $\bar\Psi$-bundles is problematic. 

We separate these problems into two issues.  The first strictly analytic issue
is to find a large subalgebra of $\glh$ with an extension of the leading order symbol trace.  Our solution in Proposition \ref{prop extend}
in fact acts on a quotient algebra of a subalgebra of $\glh$.   
This leads
to 
a different version of $\bar\Psi$ such that $\bar\Psi$-bundles with connection have 
a good theory of characteristic classes (see Definition \ref{def}).  However, the existence
of a universal $\bar\Psi$-bundle with connection is unclear, so we cannot use this theory to 
detect elements in $H^*(B\pdo_0^*,\BbC).$

The second issue is to find a Lie algebra $\calg$ such that $\pdoo$ surjects onto 
$\calg$ and such that the corresponding classifying space $BG$ is a manifold.  In fact,
we can reinterpret well known results to show that $\calg = H^{s_0}\Gamma{\rm End}(\pi^*E)$ works for $E\stackrel{\pi}{\to} N$.  
Since the leading order symbol trace extends to $\calg$, we can
define characteristic classes of $\pdo^*_0$-bundles to be pullbacks of the leading order
symbol classes of $EG\to BG$ (Definition \ref{def big}).  This approach allows us to 
detect  elements in $H^*(B\pdo_0^*,\BbC)$ (Theorem \ref{last theorem}).

In this section, we discuss analytic questions related to extensions of the leading order symbol trace.  In \S4, we discuss the 
topological questions related to the second issue.
\medskip

To begin the analysis of the first issue, we
first check that at the Lie algebra level, $\pdoo$ embeds continuously in $\gl(\calH)$, a result 
probably already known.
For a fixed choice of a finite precompact cover $\{U_\ell\}$ of $N$ and a subordinate partition of unity $\{\phi_\ell\},$ 
we write $A\in \pdoo$ as 
$A = A^1 + A^0,$ where $A^1 =
\sum_{j,k}' \phi_jA\phi_k$, with the sum over  $j, k$ with ${\rm supp}
\ \phi_j\cap
{\rm supp}\ \phi_k \neq \emptyset$, and $A^0 = A-A^1.$
Then $A^1$ is properly supported and has the classical local symbol 
$\sigma(\phi_jA\phi_k)$ in $U_j$,
and $A^0$ has a smooth kernel $k(x,y)$ \cite[Prop.~18.1.22]{hor}.  
The Fr\'echet topology on the classical
$\pdo$s of nonpositive integer order is given locally by the family of seminorms 
$$   
\sup_{x,\xi} |\partial_x^\beta\partial_\xi^\alpha\sigma(\phi_jA\phi_k)
(x,\xi)|(1+|\xi|)^{|\alpha|},  
$$ 
\begin{equation}\label{frechet}\sup_{x, |\xi|=1}|\partial_x^\beta\partial_\xi^\alpha \sigma_{-m}(\phi_jA\phi_k)(x,\xi)|,
\end{equation}
$$\sup_{x,\xi} |\partial_x^\beta
\partial_\xi^\alpha(\sigma(\phi_jA\phi_k)(x,\xi) - \psi(\xi)\sum_{m=0}^{T-1}\sigma_{-m}(\phi_j 
A\phi_k)(x,\xi))
|(1+|\xi|)^{|\alpha| +T},$$     
$$ \sup_{x,y}|\partial_x^\alpha\partial_y^\beta k(x,y)|,$$
where $\psi$ is a smooth function vanishing near zero and identically one outside a small ball centered at the origin \cite[\S18.1]{hor}.  The topology is independent of the choices of $\psi$, $\{U_\ell\}$, and $\{\phi_\ell\}.$
  Since elements $A$ of the gauge group $\calG$ of $E$ are order zero multiplication
operators with $\sigma_0(A)(x,\xi)$ independent of $\xi$, the gauge group inherits the 
usual $C^\infty$ Fr\'echet topology in $x$.   

\begin{lemma} \label{cont-incl} For the Fr\'echet topology on $\pdoo$ and the norm topology on $\gl(\calH)$, the 
inclusion $\pdoo \to\gl(\calH)$ is continuous.
\end{lemma}

\begin{proof} 
We follow \cite[Lemma 1.2.1]{gilkey}.  Let $A_i\to 0$ in $\pdo_{\leq 0}.$  Since $H^{s_0}$ is isometric 
to $L^2 = H^0$ for any $s_0$, it suffices to show that $\Vert A_i\Vert\to 0$ for $A_i:L^2\to L^2.$

As usual, the computations reduce to estimates in local charts.  We abuse notation by 
writing $\sigma(\phi_jA_i\phi_k)$ as $\sigma(A^1_i) = a_i.$
Then
$$\widehat{A^1_if}(\zeta) = \int e^{ix\cdot(\xi-\zeta)} a_i(x,\xi)\hat f(\xi) d\xi dx
= \int q_i(\zeta-\xi, \xi)\hat f(\xi) d\xi$$
for $q_i(\zeta, \xi) = \int  e^{-ix\cdot \zeta} a_i(x,\xi) dx.$ (We are using a normalized version of $dx$ in the Fourier transform.)  By \cite[Lemma 1.1.6]{gilkey},
$|A^1_i f|_0 = \sup_g \frac{|(A^1_if,g)|}{|g|_0}$, where $g$ is a Schwarz function and we use the
$L^2$ inner product.  By Cauchy-Schwarz, 
$$|(A^1_if,g)| \leq
 \left(\int |q_i(\zeta-\xi,\xi)| \ |\hat f(\xi)|^2 d\zeta d\xi\right)^{1/2}
 \left(\int  |q_i(\zeta-\xi,\xi)| \ |\hat g(\xi)|^2 d\zeta d\xi\right)^{1/2}.$$
 
 We claim that \begin{equation}\label{est}
 |q_i(\zeta-\xi, \xi)| \leq C_i, \int |q_i(\zeta-\xi, \xi)| d\xi \leq C'_i,
 \int |q_i(\zeta-\xi, \xi)| d\zeta \leq C''_i,
 \end{equation}
  with $C_i, C'_i, C''_i\to 0$ as $i\to\infty.$  If so,
 $|(A^1_if,g)| \leq D_i |f|_0\ |g|_0$ with $D_i\to 0$, and so $\Vert A^0_i\Vert \to 0.$
  
For the claim, we know $|\partial_x^\alpha a_i(x,\xi)| \leq C_{\alpha, i}$ with $C_{\alpha,i}\to 0.$  Since $a_i$ has compact $x$ support, we get 
$$|\zeta^\alpha q_i(\zeta, \xi)| = \left| \int e^{-ix\cdot \zeta} \partial^\alpha_x a_i(x,\xi) dx\right|
\leq C_{\alpha,i}$$ for a different constant decreasing to zero. In particular, 
$|q_i(\zeta, \xi)| \leq C_i(1+|\xi|)^{-1-n/2}$, for $n = {\rm dim}(N).$  This shows that 
$q_i(\zeta-\xi, \xi)$ is integrable and  satisfies (\ref{est}).

It is straightforward to show that $A^0_i\to 0$ in the Fr\'echet topology on smooth kernels
implies $\Vert A^0_i\Vert \to 0.$  Thus $\Vert A_i\Vert \leq \Vert A^0_i\Vert + \Vert A^1_i\Vert \to 0.$
\end{proof}

In order
to extend the leading order Chern class to  $\bpsi$-bundles, 
we associate an operator to the symbol of $A \in\pdoo.$   Set 
$$\opa (f)(x) =
 \sum_{j,k}{}' \int _{U_j} e^{i(x-y)\cdot\xi} \sigma(\phi_j A\phi_k)(x.\xi) f(y)dyd\xi
 $$
Then $A - \opa\in\pdomi$, the closed ideal of $\pdo$s of order $-\infty$,  and 
$\sigma(A) \stackrel{\rm def}{=} \sigma(A^1) = \sigma(\opa).$  Note that $\opa$ is 
shorthand for the $\pdo$ ${\rm Op}(\sigma(A))$
 noncanonically associated to $\sigma(A) \in  \Gamma({\rm End}(\pi^*E)\to S^*N).$

\begin{definition} $$\op = \{\opa: A\in \pdoo\}$$
\end{definition}

We emphasize that $\op$ depends on a fixed atlas and subordinate partition of unity for 
$N$. 
 The closed vector space $\op$ is not an algebra, but the linear map $o:\pdoo\to\op, A\mapsto \opa$ is continuous.
 Let $\bop$ be the closure of $\op$ in $\glh$.
 
Fix $K>0$, and set 
\begin{eqnarray} \label{decay} \op^K &=& \{ \opa: |\partial^\alpha_\xi (\sigma(A) - \sigma_0(A))(x,\xi)| 
\leq K(1 + |\xi|) ^{-1},  \\
&&\quad  |\partial_\xi^\alpha \sigma_0(A)(x,\xi)| \leq K,
\ {\rm for}\ 
|\alpha| \leq 1, 
\forall (x,\xi)\in T^*U_j, 
\forall j\}.\nonumber
\end{eqnarray}
Since $\sigma_0$ has homogeneity zero and $S^*N$ is compact, every $\opa\in \op$ 
lies in some $\op^K.$

\begin{lemma}  \label{key lemma} $A\mapsto \int_{S^*N}\tr\ \sigma_0(A)(n,\xi) d\xi\ \dvol(n)$ extends from a continuous map 
 on $\op^K$ to a continuous map on $\overline{\op^K}.$
\end{lemma}

\begin{proof}

We must show that if $\{{\rm Op}_1(A_i)\}\subset \op^K$ is Cauchy in the norm topology on 
${\rm End}(\calH)$,  then $\{\tr\ \sigma(A_i)\}$ is Cauchy in $L^1(S^*N).$  Fix a finite
cover $\{U_\ell\}$ of $N$ with $\bar U_\ell$ compact.
 The hypothesis is 
$$\left\Vert \int_{T^*U_k}e^{i n\cdot \xi}\phi_\ell(n)(\sigma(A_i) -\sigma(A_j))(n,\xi)
\phi_k(n) \widehat{ f}(\xi)\ d\xi\right\Vert_{s_0} <
\epsilon \Vert f\Vert_{s_0}$$
for $i, j > N(\epsilon),$ and for each $\ell,k$ with ${\rm supp}\ \phi_\ell \cap {\rm supp}\ 
\phi_k\neq \emptyset.$
  Sobolev embedding implies that
\begin{equation}\label{gij}
n\mapsto g_{i,j}(n) = \frac{1}{ \Vert f\Vert_{s_0}}\int_{T_n^*U_k} e^{in\cdot \xi}\phi_\ell(n)
(\sigma(A_i) -\sigma(A_j))(n,\xi)\phi_k(n) \widehat{  f}(\xi)\ d\xi
\end{equation}
is $\epsilon$-small in $C^r(U_k)$ for any $r < s_0 - ({\rm dim}\ N)/2$ and any fixed   $f\in H^{s_0}.$

 Fix $U_\ell$, and pick $\xi_0$ in the cotangent space of a point $n_1\in U_\ell$
 with $\phi_\ell(n_1)\phi_k(n_1)$\\
 $\neq 0$. We can
 identify all cotangent spaces in $T^*U_\ell$ with $T^*_{n_1}U_\ell.$  We claim that
$$h_{i,j}(n, \xi_0) =\phi_\ell(n)( \sigma(A_i) -\sigma(A_j))(n,\xi_0)\phi_k(n)$$
has $|h_{i,j}(n,\xi_0)|<\epsilon$  for all $n\in U_k$, for all $\ell$, and for 
$i, j \gg 0$.  Since $\sigma_0$ has
homogeneity zero, we may assume that $\xi_0\in S^*N.$
 Thus we claim that \\
 $\{\phi_\ell(n)\sigma(A_i)(n, \xi_0)\phi_k(n)\}$ is 
Cauchy in this fixed chart, and so by compactness
$\{\sum_{\ell, k}'\sigma(\phi_\ell A_i\phi_k) = \sigma(A_i)\}$ will be  uniformly Cauchy on all of $S^*N.$

For the moment, we assume that our symbols are scalar valued.
 If
the claim fails,
 then by compactness of $N$ there exists $n_0$ and $\epsilon >0$ such that there exist $i_k, j_k \to\infty$ with $|h_{i_k j_k}(n_0,\xi_0)| > \epsilon.$
Let $\widehat {f} = \widehat {f}_{\xi_0, \delta}$ be a 
bump function of height 
$  
e^{-in_0\cdot\xi_0}$  concentrated on $B_{r(\delta)}(\xi_0)$, the metric ball in $T_{n_0}^*U_\ell$ centered at $\xi_0$ and of volume $\delta$,
and let $b_{\xi_0, \delta}$ be the corresponding bump function of height one.
 Taylor's theorem in the form 
 \begin{equation}\label{taylor}
 a(\xi_0) - a(\xi) = - \sum_\ell (\xi-\xi_0)^k\int_0^1 \partial^\ell_\xi a(\xi_0 + t(\xi-\xi_0)) dt
 \end{equation}
 applied to $a(\xi) =  h_{i_kj_k}(n_0,\xi)$
 implies
 \begin{eqnarray*} 
 \lefteqn{
 \left|  \int_{T_{n_0}^*U_\ell}  e^{in_0(\xi-\xi_0)} h_{i_kj_k}
 (n_0, \xi)b_{\xi_0, \delta} d\xi \right| }\\ 
 &\geq &
 \left|  h_{i_kj_k}
 (n_0, \xi_0) \int_{T_{n_0}^*U_\ell} e^{in_0\xi} b_{\xi_0, \delta} 
 d\xi\right| \\
 &&\qquad -\left|    \int_{T_{n_0}^*U_\ell}  e^{in_0(\xi-\xi_0)}( h_{i_kj_k} (n_0, \xi_0) - h_{i_kj_k}(n_0,\xi)) b_{\xi_0, \delta} d\xi\right|\\
 &\geq & \frac{1}{2} h_{i_kj_k}(n_0,\xi_0) \delta -  r(\delta) F(\delta, (n_0, \xi_0))\delta,
  \end{eqnarray*}
for some $F(\delta,(n_0,\xi_0) )\to 0$ as $\delta \to 0$.  To
produce this $F$,  we use $|(\xi-\xi_0)^k|\leq r(\delta)$
and (\ref{decay}) with $|\alpha| = 1$ to bound the partial derivatives of 
$ h_{i_kj_k}(n_0,\xi)$ in (\ref{taylor}) by a constant independent of $i_k, j_k.$  
 For $\delta$ small enough, 
 $ r(\delta)F(\delta, (n_0, \xi_0))) < \frac{1}{4}|h_{i_kj_k}(n_0, \xi_0)|$
 for all $k$, and so 
 \begin{equation}\label{hij}
\left| \int_{T_{n_0}^*U_\ell} 
 e^{in_0(\xi-\xi_0)} h_{i_kj_k}
 (n_0, \xi)b_{\xi_0, \delta} d\xi \right| > \frac{1}{4}|h_{i_kj_k}(n_0, \xi_0)|\delta
 \end{equation} 
 
 Similarly,  with some abuse of notation,  we have
 $$\Vert f\Vert_{s_0} =
  \left( \sum_\ell \int_{U_\ell} 
  (1+|\xi|^2)^{s_0} |\widehat{ (\phi_\ell\cdot f)}(\xi)|^2 d\xi\right)^{1/2}.$$
  Since $|\widehat{ (\phi_\ell\cdot f)}(\xi)| 
\leq  C\cdot \hat \phi_\ell(\xi-\xi_0)\delta$ for some constant $C$ which we can take independent of $\delta$ for small $\delta$, we get
 \begin{eqnarray}\label{denom} 
 \Vert f\Vert_{s_0} & \leq & \left( \Sigma_\ell \int_{U_\ell} (1+|\xi|^2)^{s_0} C \delta^2|\hat\phi_\ell(\xi-\xi_0)|^2
 \right)^{1/2}\\
 &\leq& C\delta\nonumber
 \end{eqnarray}
 where $C$ changes from line to line.  
 
 Substituting (\ref{hij}) and (\ref{denom}) into (\ref{gij}), we obtain 
 $$ |g_{i_kj_k}(n_0)| \geq C|h_{i_kj_k}(n_0, \xi_0)| \geq C\epsilon$$
 for all $k$, a contradiction.  Thus $h_{ij}(n, \xi_0)$ has the claimed estimate.
 
 If the symbol is matrix valued, we replace the bump functions by sections of the 
 bundle $E$ having the $r^{\rm th}$ 
coordinate in some local chart given by the bump functions and all other coordinates zero.  
The argument above shows that 
 the $r^{\rm th}$ columns of $\sigma(A_i)$ form a Cauchy sequence, and so each sequence
 of entries $\{\sigma(A_i)^s_r\}$ is Cauchy.  
 
 If $\{\tr(\sigma_0(A_i))\}$ is not uniformly Cauchy on $S^*N$, then there exists $\epsilon >0$ 
with an infinite sequence of $i, j$ and 
 $(n, \xi)\in S^*N$ such that 
\begin{eqnarray*} \lefteqn{
|\tr(\sigma(A_i))(n , \lambda\xi) - \tr(\sigma(A_j))(n , \lambda\xi)| }\\
&=& 
|\tr((\sigma-\sigma_0)(A_i))(n , \lambda\xi) + \tr\ \sigma_0(A_i))(n , \lambda\xi)\\
&&\qquad 
-\tr((\sigma-\sigma_0)(A_j))(n , \lambda\xi) - \tr\ \sigma_0(A_j))(n , \lambda\xi)|\\
&\geq & |\tr(\sigma_0(A_i))(n , \xi) - \tr\ \sigma_0(A_j))(n , \xi)| \\
&&\qquad 
- |\tr((\sigma-\sigma_0)(A_i))(n , \lambda\xi) 
- \tr((\sigma-\sigma_0)(A_j))(n , \lambda\xi)|\\
&\geq & \epsilon - 2K(1+|\lambda|)^{-1}
\end{eqnarray*}
for all $\lambda >0.$  For $\lambda 
\gg 0$ this contradicts that $h_{i,j}$ is Cauchy.

 This implies that
$\{\tr(\sigma_0(A_i))\}$ is uniformly Cauchy on $S^*N$, so the claimed extension exists.
The continuity of the extension is immediate. 
\end{proof}

Fix $K$, and set
$$\calA =\calA^{K} =  \cup_{n\in \BbZ^+} \overline{\op^{nK}} \subset \gl(\calH),$$
where the closure is taken in the norm topology. 
Then
$$\bar \calA = \overline{\cup_{n\in \BbZ^+} \overline{\op^{nK}} }
= \overline {\cup_{n\in \BbZ^+} \op^{nK} } = \overline {\op} \subset \gl({\calH})$$
is independent of $K$.

\begin{corollary}  $A\mapsto \int_{S^*N}\tr\ \sigma_0(A)(n,\xi) d\xi\ \dvol(n)$ extends from a continuous map 
 on $\op$ to a continuous map on $\overline{\op} = \bar\calA.$
\end{corollary}

\begin{proof}  Fix $K$. We first show that the leading order symbol extends to $\calA.$ 
For $n'>n,$ the inclusion $i_{n,n'}:\overline{\op^{nK}}\to \overline{
\op^{n'K}}$ has
$\sigma^{n'K}\circ i_{n,n'} = \sigma^{nK}$ for the extensions $\sigma^{nK}, \sigma^{n'K}$ on 
$\overline{\op^{nK}}, \overline{\op^{nK'}}$.
Thus for any $K$, we
 can unambiguously set $\sigma(A) = \sigma^{nK}(A)$ for $A\in \overline{\op^{nK}}$.  
 The continuity follows from the previous lemma.   Since the extension is linear, it is immediate that the extension is infinitely
 Fr\'echet differentiable on $\calA$ inside the Banach space $\gl(\calH).$
 
 Since $\sigma$ is continuous on $\calA$, it extends to a continuous linear functional on $\bar\calA$, which is again smooth.  
 \end{proof}

To discuss the tracial properties of extensions of $\int_{S^*N}\tr\ \sigma_0$, we must have
algebras of operators. 
Since $\op(AB) - \op(A)\op(B)\in \pdom$, $\op/\pdom$ (i.e. $\op/(\pdom\cap\op)$)
is an algebra.  Since $A-\opa\in\pdom$, we have $\op/\pdom\simeq \pdoo/\pdom.$  
Note that if $\op'(A)$ is defined as for
$\op(A)$ but with respect to a different
atlas and partition of unity, then
$\op(A) - \op'(A)$  is a smoothing operator.  Thus $\op/\pdom$ is canonically defined, independent of these choices. 

Let $\bop$ be the closure of $\op$ in $\glh$, and let
$\calC$ denote the closure  of $\pdomi$ in $\bop.$  $\calC$ is easily a closed ideal in $\bop.$
On quotients of normed algebras, we take 
the quotient norm $\Vert [A]\Vert = \inf\{\Vert A\Vert: A\in [A]\}.$

\begin{proposition}\label{prop extend}  $A\mapsto \int_{S^*N}\tr\ \sigma_0(A)(n,\xi) d\xi\ \dvol(n)$ extends from a continuous trace  
 on $\pdoo/\pdom$ to a continuous trace $\sigma$ on $\bop/\calC\subset \glh/\calC.$
 \end{proposition}

\begin{proof}  Since $\pdom\subset \calC$, it is immediate that the leading order symbol integral descends to a continuous
trace on $\pdoo/\pdom\simeq \op/\pdom$ and extends to a continuous map on $\bop/\calC$.  
To see that the extension is a trace, 
 take $A, B\in \bop/\calC$ and $A_i, B_i\in \op/\pdom$ with 
 $A_i\to A, B_i\to B$.  The leading order symbol
 trace
 vanishes on $[A_i, B_i]$, so by continuity $\sigma([A,B]) = 0.$   
\end{proof}

We can pass from the Lie algebra to the Lie group level via the commutative diagram
\begin{equation}\label{cd}
\begin{CD}
\Gamma{\rm End}(E) @>>> \pdo_{\leq 0} @>>> \frac{\pdo_{\leq 0}}{\pdom} \simeq
\frac{\op}{\pdom} @>>>  \frac{\bop}{\calC} \\
@V\exp VV     @V\exp VV @V\exp VV @V\exp VV\\
\calG @>>>\pdo_0^*@>>>     \frac{ \pdo_0^*}{(I + \pdom)^*} @.
\frac{\exp(\bop)}{(I+\calC)^* }
\end{CD}
\end{equation}
 $\calG$ is 
 the gauge group of $E\stackrel{\pi}{\to} N$, $(I + \pdom)^*$
 refers to invertible operators $I + B, B\in \pdom$, and similarly for other
 groups on the bottom line. 
 The diagram consists of
 continuous maps if the spaces in the first three columns have either the norm or 
 the Fr\'echet topology and the spaces in the last  column have the norm topology.
 The exponential map is clearly surjective in the first column, and   by standard Banach space arguments, the exponential map is surjective in the fourth  column.  The surjectivity of the second and third column is not obvious.  In particular, there is no obvious map from $
  \frac{ \pdo_0^*}{(I + \pdom)^*}$ to 
$\frac{\exp(\bop)}{(I+\calC)^* }$.
 In any case,
the maps on the bottom line are group  
 homomorphisms. 
  
As discussed in the beginning of this section, it would be desirable to work with $\bar\Psi = \overline{\pdo_0^*}$ or the closure of the standard variant $  \frac{ \pdo_0^*}{(I + \pdom)^*} $.
 However, it seems difficult if not impossible
to extend the leading order symbol trace to the corresponding Lie algebras $\overline{\pdoo},
\overline{\pdo_{\leq 0}}/\overline{\pdomi}.$  The following re-definition is our substitute for $\overline{\pdo_0^*}$, as it is the group associated to the largest known Lie algebra with an extension
(strictly speaking, a factorization) of  the leading order symbol trace. 
 
 \begin{definition}\label{def} $\bar\Psi = \frac{\exp(\bop)}{(I+\calC)^* }$.
 \end{definition}

The smooth factorization  of the leading order symbol trace $\sigma: \pdo_{\leq 0}\to\BbC$ 
through 
$\bop/\calC$  gives us a geometric theory of characteristic classes for $\bpsi$-bundles.

\begin{theorem} \label{ext thm} The leading order Chern classes $\cklo$ extend to $\bpsi$-bundles over 
paracompact manifolds.
\end{theorem}

\begin{proof}   Such bundles $\calE\to B$  admit connections with curvature
$\Omega\in$\\
$ \Lambda^2(B, \bop/\calC).$
By the previous proposition, we may set
$$\cklo(\nabla)_b \stackrel{\rm def}{=}
\sigma(\Omega^k_b)$$
at each $b\in B.$  Since $\sigma$ is smooth, the corresponding de Rham 
$\cklo(E\bpsi)$ class is closed and 
independent of connection 
as in finite dimensions.  
\end{proof}

In summary, the leading order trace on  $\pdo_0^*$-bundles trivially factors
through $\pdo_0^*/(I+\pdom)^*$-bundles, as this quotient just removes the smoothing term
$A^0$ from an invertible $\pdo$ $A$.  By Proposition \ref{prop extend}, this
trace then extends continuously to $\bop/\calC$.  This space is morally the closure of 
$\op/\pdomi$ in $\glh/\calC$, but it is not clear that $\calC$ is an ideal in $\glh.$ In any case, the work in this section has the feel of extending a continuous trace from a set to
its closure.   

\begin{remark}  The groups in (\ref{cd}) give rise to different bundle theories, since the homotopy types of $\pdo_0^*$, and $\pdo_0^*/(I + \pdomi)^*$ differ \cite{R}; the homotopy type of $\calG$, discussed in \cite{A-B1}, is almost surely different from that of $\pdo_0^*, \pdo_0^*/(I + \pdomi)^*$.  
The relationship between the topology of $\pdo_0^*/(I + \pdomi)^*$ and $\bpsi = \exp(\bop)(I+\calC)^*$ is 
completely open, so presumably the bundle theories for these groups also differ.

Geometrically, one can construct $\bpsi$-connections by using a partition of unity to glue together trivial connections over trivializing neighborhoods in a paracompact base.  However, even for finite dimensional Lie groups, it is difficult to compute the corresponding Chern classes of such glued up connections.  As a result, we do not have examples of nontrivial leading order Chern classes for $\bpsi$-bundles.
\end{remark}

\section{Detecting cohomology of $B\pdo_0^*$}

We now turn to the second issue 
discussed in the beginning of \S3, namely that classifying spaces are not manifolds
in general.  
We could obtain information about $H^*(B\bar\Psi,\BbC)$ if $E\bar\Psi\to B\bar\Psi$  admitted a connection, but this presupposes that $B\bar\Psi$ is a manifold.  As an alternative, we consider the 
exact sequence of algebras associated to $E\stackrel{\pi}{\to} N$:
$$0\to \pdomo\to\pdoo\stackrel{\sigma_0}{\to} \Gamma{\rm End}(\pi^*E\to S^*N)\to 0,$$
which gives $\Gamma{\rm End}(\pi^*E\to S^*N)\simeq \pdoo/\pdomo.$  
Here $\pdomo$ is the algebra of classical integer order $\pdo$s of order at most $-1.$
Note that the quotient $\pdo_{\leq 0}/\pdomi$ considered in \S3 is more complicated topologically than $\pdo_{\leq 0}/\pdomo.$

We obtain the diagram
\begin{equation}
\label{diagone}
\begin{CD}
\Gamma{\rm End}(E) @>>> \pdo_{\leq 0} @>>>\Gamma {\rm End}(\pi^*E)
\\
@V\exp VV     @V\exp VV @V\exp VV \\
\calG(E) @>j>>\pdo_0^*@>m>>    \calG(\pi^*E)\end{CD}
\end{equation}

  By Lemma \ref{cont-incl}, $j$ and $m$ are continuous,
 where $\calG(E)$ has the Fr\'echet or norm topology, 
$\pdo_0^*$ 
has the Frech\'et topology, and $\calG(\pi^*E)$ has the Fr\'echet or the norm topology. 
The bottom line of this diagram induces
\begin{equation}\label{diag}
\begin{CD}
E\calG(E) @>>> E\GG@>>> E\calG(\pi^*E)\\
@VVV@VVV@VVV\\
B\calG(E)@>Bj>> B\GG @>Bm>> B\calG(\pi^*E)
\end{CD}
\end{equation}
since $E\GG \simeq (Bm)^* E\calG(\pi^*E)$ and similarly for $E\calG(E)$ by the Milnor construction.   

We can now define leading order Chern classes of $\GG$-bundles, avoiding the question
of the existence of connections on $E\GG.$  By \cite{A-B1}, for $E^\ell\to N$ and $\calG = \calG(E)$, 
$B\calG= {\rm Maps}_{0}(N,BU(\ell))=\{f:N\to BU(\ell) | f^*EU(\ell) \simeq E\}$, and 
$E\calG|_f$ is the subset of ${\rm Maps}(E, EU(\ell))$ covering $f$.   Equivalently, 
$E\calG = \pi_*\ev^*EU(\ell).$  Recall that we are using maps in a fixed large
Sobolev class;
 these maps uniformly approximate smooth maps, so that the homotopy types of 
$E\calG$ and $B\calG$ are the same for smooth maps or maps in this Sobolev
class.
Thus $B\calG(E)$ and
$B\calG(\pi^*E)$ are  Banach manifolds.   For any topological group $G$, $BG$ admits a partition of unity \cite
[Thm. 4.11.2]{huse}.  Thus $E\calG(\pi^*E)\to B\calG(\pi^*E)$ admits a connection
with curvature $\Omega\in \Lambda^2(B\calG(\pi^*E), \Gamma{\rm End}(\pi^*E)).$  Note
that the leading order symbol trace  on $\pdoo$ obviously induces a trace 
$\sigma$ on $\pdoo/\pdomo\simeq \Gamma
{\rm End}(\pi^*E).$  Therefore $E\calG(\pi^*E)\to B\calG(\pi^*E)$ has associated de Rham
classes $\cklo(E\calG(\pi^*E)) = [\sigma(\Omega^k)] \in H^{2k}(B\calG(\pi^*E), \BbC).$

\renewcommand{\bpsi}{\calG(\pi^*(E)}
The following definition is natural in light of (\ref{diag}).

\begin{definition}\label{def big}  The k${}^{\rm th}$ leading order Chern class $\cklo(E\GG)$ is the de Rham
cohomology class of $(Bm)^*\cklo(E\bpsi) \in H^{2k}(B\GG, \BbC).$
\end{definition}

Set $\calG = \calG(E)$.  Let 
$\calE\to B$ be a $\calG$-bundle, classified by a map $g:B\to B\calG$.    
The maps $j$ and  $m\circ j$ in (\ref{diagone}) are injective, so every $\calG(E)$-bundle is
both a $\pdo_0^*(E)$-bundle and a
$\calG(\pi^*E)$-bundle. We get 
$$c_k^{lo}(\calE) = g^*c_k^{lo}(E\calG) = g^*Bj^*c_k^{lo}(E\GG) =  g^*Bj^*Bm^*\cklo(E\bpsi).$$
This gives an easy criterion to detect cohomology classes for the classifying spaces.

\begin{lemma}\label{lem:one}  Let $\calE$ be a $\calG$-bundle with $\cklo(\calE) \neq 0.$ Then the 
cohomology classes 
$\cklo(E\calG)\in H^{2k}(B\calG, \BbC),\ \cklo(E\GG)\in H^{2k}(B\GG, \BbC), 
\cklo(E\bpsi)\in $\\
$H^{2k}(B\bpsi, \BbC)$
are all nonzero.
\end{lemma}

As in Remark 2.1, let $\pi: {\rm Maps}(N,M) \times N\to
{\rm Maps}(N,M)$ be the projection, and for
$n\in N$, define $\ev_{n}: {\rm Maps}(N,M)\to M$ by $\ev_n(f) = \ev(n,f) =  f(n).$

\begin{example}  \label{example}
Let $F\to M$ be a complex bundle, and set $E = \ev^*F\to \mapsfnm \times M$,
$ \calE = \pi_*\ev^*F \to \mapsfnm.$  Here $\mapsfnm$ is the component of a fixed $f:N\to M.$
Then the Lemma applies with $\calG = \calG(f^*F)$, since $\calE_g$ is noncanonically
isomorphic to $H^{s_0}\Gamma(f^*F)$ for all $g\in \mapsfnm.$  
\end{example}

\begin{lemma} \label{lem:ConnEG}
$E\calG$ has a universal connection with connection one-form
$\theta^{E\calG}$ defined on $s\in\Gamma(E\calG)$ by
\begin{equation*}
(\theta^{E\calG}_Z s)(\gamma)
(\alpha)=\left( (\ev^*\theta^u)_{(Z,0)} u_s\right)(\gamma,\alpha).
\end{equation*}
Here $\theta^u$ is the universal connection on $EU(k))\to BU(k)$, and
$u_s: {\rm Maps}(N,M)\times N \to \ev^*EU(k)$ is defined by
$u_s(f,n)=s(f)(n)$.

\end{lemma}

\begin{corollary}
The curvature $\Omega^{E\calG}$ of $\theta^{E\calG}$ satisfies
\begin{equation*}\label{eq:pullback}
\Omega^{E\calG}(Z,W)s(f)(n)=\ev^*\Omega^u  ((Z,0),(W,0)) u_s(f,n).
\end{equation*}
\end{corollary}

The proofs are in \cite[\S4]{P-R} for loop spaces ($N = S^1$) and easily extend.
As a result, the leading order Chern classes of gauge bundles are pullbacks of finite dimensional
classes.

\begin{lemma}\label{eq:IO}
Fix $n_0\in N$.  Then
$$\cklo(E\calG) = {\rm vol}(S^*N)\cdot \ev_{n_0}^* c_k(EU(\ell)).$$
If $\calF\to \mapsnmone$ is given by $\calF = \pi_*\ev^*F$ for a bundle $F\to M$,
 then  $$\cklo(\calF)
= {\rm vol}(S^*N)\cdot \ev_{n_0}^* c_k(F).$$
\end{lemma}

\begin{proof}
For all $n_0\in N$, the maps
${\rm ev}_{n_0} $ are homotopic, so the de Rham class of
\begin{equation*}\label{eq:Inde}
\int_{S^*N}\tr\ \sigma_0(\ev_{n_0}^*(\Omega^u)^k) \ d\xi\ \dvol(n_0) = 
 {\rm vol}(S^*N)\cdot \tr\ \sigma_0(\ev_{n_0}^* (\Omega^u)^k)
\end{equation*}
is
independent of $n_0$.  
Since $\Omega^u$ is a multiplication operator, we get
$$\cklo(E\calG) ={\rm vol}(S^*N)\cdot \left[\ev_{n_0}^* (\tr\ \Omega^u)^k\right] = 
{\rm vol}(S^*N)\cdot \ev_{n_0}^* c_k(EU(\ell)).$$
The proof for $\calF$ is identical.
\end{proof}

Note that $({\rm vol}(S^*N)\cdot (2\pi i)^k)\cklo(E\calG) \in H^{2k}(B\calG, \BbZ)$.  It is 
not clear that we can normalize $\cklo(E\GG), \cklo(E\bpsi)$ to be integer classes.

We can produce examples of nontrivial $\cklo({\rm Maps}(N,M) = \cklo(T{\rm Maps}(N,M))$ as well
as other cohomology classes for $\mapsnmone.$

\begin{theorem}\label{4.7}
(i)  Let $M$ have $c_k(M) = c_k(TM\otimes \BbC)\neq 0$.  Then
 $$0\neq \cklo(\mapsnmone) \in H^{2k}(\mapsnmone, \BbC).$$

(ii) Let $F^\ell\to M$ be a finite rank Hermitian bundle with $c_k(F)\neq 0.$  Then
$$0\neq \cklo(\pi_*\ev^*F)\in H^{2k}(\mapsnmone, \BbC).$$

 \end{theorem}
 
\begin{proof}
 (i) 
 Let $h:M\to BU(m)$ classify $TM\otimes \BbC.$  (Strictly speaking, we take a classifying
 map into a Grassmannian $BU(m, K)$ of $m$-planes in $\BbC^K$ for $K\gg 0$, so that the target space is a finite dimensional manifold.)
 For fixed $f\in\mapsnmone$, the gauge group $\calG$ of $f^*(TM\otimes \BbC)$ has 
\begin{eqnarray*}B\calG &=& \{g\in {\rm Maps}(N, BU(m)): g^*BU(m) \simeq f^*(TM\otimes \BbC)\} \\
&=&  \{g\in {\rm Maps}(N, BU(m)): g\sim hf\}.
\end{eqnarray*}
Therefore the map
$$\tilde h:\mapsnmone\to {\rm Maps}(N, BU(m)),\ \tilde h(f) = hf$$
classifies $T\mapsnmone,$  and so for fixed $n\in N$, 
$$\cklo(\mapsnmone) = {\rm vol}(S^*N)\cdot \tilde h^*\ev_n^* c_k(EU(m)),$$
by Lemma \ref{eq:IO}.
 
 Let $[a] $ be a $2k$-cycle with $\langle c_k(M),[a]\rangle \neq 0.$  (Here the bracket refers
 to integration of forms over cycles.)
 For $i:M\to \mapsnmone, i(m_0)(n) = m_0$, set
$[\tilde a] = i_*[a].$  
Then
$$\langle \cklo(\mapsnmone), [\tilde a]\rangle =  {\rm vol}(S^*N)\cdot\langle c_k(EU(m)), \ev_{n,*}
\tilde h_*[\tilde a]\rangle.$$
It is immediate that $\ev_n \tilde h i = h$, so 
$\ev_{n,*} \tilde h_*[\tilde a] = h_*[a].$  Therefore
\begin{eqnarray*}\langle \cklo(\mapsnmone), [\tilde a]\rangle &=&  {\rm vol}(S^*N)\cdot\langle
h^* c_k(EU(m)),[a]\rangle \\
&=&  {\rm vol}(S^*N)\cdot\langle c_k(TM\otimes \BbC), [a]\rangle\\
& \neq& 0.
\end{eqnarray*}

(ii)  Let $h:M\to BU(\ell)$ classify $F$.  As above, $\tilde h$ classifies $\calF = \pi_*\ev^*F$.  Thus
$$\cklo(\calF) = \tilde h^* \cklo(E\calG) = {\rm vol}(S^*N)\cdot \tilde h^*\ev_n^*c_k(EU(\ell))$$
by Lemma \ref{eq:IO}.  As in (i), we get
$$\langle \cklo(\calF), [\tilde a]\rangle = \langle c_k(F), [a]\rangle \neq 0$$
for some cycle $[a]$.  Alternatively, we can use the last statement in Lemma 3.13 and 
$\ev_n i = {\rm Id}$ to reach the same conclusion.
\end{proof}

In  this proof,
the cycle $[\tilde a]$ has image in ${\rm Maps}_c(N,M)$, the component of the constant maps, so the result is really about bundles over this component.  We can improve this to cover all components.

\begin{corollary} \label{cor3} For $f\in\mapsnmone$, let $\mapsfnm$ denote the connected component 
of $f$.  Let $F^\ell\to M$ be a finite rank Hermitian bundle with $c_k(F)\neq 0.$  Assume that $M$ is connected. Then
$$0\neq \cklo(\pi_*\ev^*F)\in H^{2k}(\mapsfnm, \BbC).$$
\end{corollary}

\begin{proof}
We claim that for fixed $n_0\in N$ the map 
$$\ev_{n_0, *}:H_*(\mapsfnm, \BbC) \to H_*(M,\BbC)$$
  is 
surjective.  

As a first step, we show that for a fixed $m_0\in M$, we can homotop $f$ to a map $\tilde f$ with $\tilde f(n_0) = m_0$.  By the tubular neighborhood theorem applied 
to the one-manifold/path from $f(n_0)$ to $m_0$, there exists a coordinate chart $W = \phi(\BbR^n)$ 
containing $f(n_0)$ and $m_0.$  Take small coordinate balls $U$ containing $\phi^{-1}(
f(n_0))$ and $V$
containing $\phi^{-1}(m_0)$, such that  $V$ is a translate $\vec T + U$  of 
 with $\vec T = \phi^{-1}(m_0) - \phi^{-1}(n_0)$.  We may assume that $U$ is a ball of radius $r$ centered at $\phi^{-1}(n_0)$.  Let $\psi:[0,r]\to \BbR$ be a nonnegative bump
 functions which is one near zero and zero near $r$.  
Define $f_t:N\to M$ by 
 $$f_t(n) = \left\{ \begin{array}{ll} f(n),  &
\ \ f(n)\not \in \phi(U),\\
                   \phi[ (1-t) \phi^{-1}(f(n))&\\
\ \                     + t\psi(d(\phi^{-1}(f(n)), \phi^{-1}(f(n_0)))(\vec T + \phi^{-1}(f(n))]
&                   \ \  f(n)\in \phi(U). \end{array}\right.    $$
In other words, $f_t = f$ outside $f^{-1}(\phi(U))$ and moves points $f(n)\in \phi(U)$ towards $\phi(V)$, with $f_t(n_0)$ moving $f(n_0)$ to            
$m_0$.  Now set $\tilde f = f_1.$

Take a $k$-cycle $\sum r_i \sigma_i$ in $M$.  Let $\Delta^k = \{(x_1,\ldots, x_k): x_i\geq 0, \sum x_i \leq 1\}$ be the standard $k$-simplex.
By subdivision, we may assume that each $\sigma_i(\Delta^k)$ is in 
a coordinate patch $V'= V'_i = \phi(V_i)$ in the notation above.  
Construct the corresponding neighborhood $U' = \phi(U_i)$ of 
$f(n_0).$ Set $m_0 = \sigma_i(\vec 0)$. Take
a  map $\alpha_i:\Delta^k\to\mapsfnm$ with $\alpha(\vec 0) = f$ and $\alpha_i(x)(n_0)\in U'$
for all $x\in \Delta^k$.  By suitably modifying the bump function to vanish near
$d(\phi^{-1}(\alpha_i(x)(n_0)), \partial U)$, we can form a simplex $\tilde \alpha_i:
\Delta^k\to\mapsfnm$ with $\tilde\alpha_i(x)(n_0) = \sigma_i(x)$  and $\tilde\alpha_i(\vec 0)
= \tilde f.$

Clearly $\sum r_i\tilde\alpha_i$ is a cycle in $\mapsfnm$ with $\ev_{n_0,*} [\sum r_i\tilde\alpha_i] = 
[\sum_i r_i\sigma_i].$  This finishes the claim. Note that this construction is an {\it ad hoc} replacement for the map $i$ in the last theorem.  

Pick a cycle $[b]\in H_{2k}(M, \BbC)$ such that $\langle c_k(F), [b]\rangle \neq 0.$  Pick 
$[\tilde b]\in $\\
$H_{2k}(\mapsfnm, \BbC$ with $\ev_{n_0,*}[\tilde b] = [b].$  Then
$$ \langle \cklo(\pi_*\ev^*F), [\tilde b]\rangle =  
{\rm vol}(S^*N)\cdot\langle c_k(F), \ev_{n_0,*}[b]\rangle \neq 0$$
as in Theorem \ref{4.7}(ii).
\end{proof}

This gives information about the cohomology rings of the various classifying spaces.  
Recall that we are working with either the Fr\'echet or the norm topology on $\GG$.

\begin{proposition}\label{last prop} Fix a closed manifold $N$ and a 
connected manifold $M$. 
Let $F^\ell\to M$ be a finite rank Hermitian bundle, choose $f:N\to M$, and
let $\calG, \GG,$ refer to 
the gauge groups and $\pdo$ groups acting on sections of $F$.  Let $H^*_{F}(M,\BbC)$ be 
the subring of $H^*(M,\BbC)$ generated by the Chern classes of $F$.  
Then for $X = B\calG, B\GG, B\calG(\pi^*F)$, there is a surjective map from 
$H^{*}(X,\BbC)$ to  an isomorphic copy of $H^*_{F}(M,\BbC)$ in 
$H^*( {\rm Maps}_f(N, M),\BbC)$, where $\mapsfnm$ is the component of $f$ in $\mapsnmone.$
\end{proposition}

\begin{proof}  
Set $\calF = \pi_*ev^*F.$  The proof of Corollary \ref{cor3} shows that
if a polynomial
 $p(c_0^{\rm lo}(\calF), ..., c_\ell^{\rm lo}(\calF))\in H^*(\mapsfnm, \BbC)$ vanishes, then
$p(c_0(F),...,c_\ell(F))$\\
$ = 0.$  Thus $H_F^*(M,\BbC)$ injects
 into $H^*(\mapsfnm, \BbC)$, for any $N$, with image the ring generated by 
 $\cklo(\calF)$, $k\leq \ell.$
   Let $h$ classify $F$.  In the notation of (\ref{diag}) and the previous 
 theorem, we have 
 $$\cklo(\calF) = (Bj\circ Bm\circ \tilde h)^*\cklo(E\calG(\pi^*F)) = (Bm\circ \tilde h)^*\cklo(E\GG)
 = \tilde h^*\cklo(E\calG).$$
 The surjectivity of $H^*(X,\BbC) \to {\rm Im}(H_F^*(M,\BbC))$ is now immediate.
\end{proof}  

This gives the result on the cohomology of $B\calG, B\GG, B\bpsi$ stated in the Introduction.

\begin{theorem}\label{last theorem}
Let $E^\ell\to N$ be a finite rank Hermitian bundle, and let 
$\calG, \GG,$
refer to 
the gauge groups and $\pdo$ groups acting on sections of $E$.  Then 
for $X = B\calG, B\GG, B\calG(\pi^*E)$, there is a surjective map from 
$H^{*}(X,\BbC)$ to  the polynomial algebra $H^*(BU(\ell), \BbC) = \BbC[c_1(EU(\ell)),\ldots, c_\ell(EU(\ell))].$
\end{theorem}

\begin{proof}
Let $M = BU(\ell, K)$ be the Grassmannian of $\ell$-planes in $\BbC^K$, for $K \gg 0,$
let  $F = EU(\ell, K),$ and let
$f:N\to M$ classify $E$.  On the component $\mapsfnm$ of $f$, $\calE = \pi_*\ev^*EU(\ell,K)$
has structure group $\calG(f^*EU(\ell,K)) = \calG(E).$   $H^*(M,\BbC)$ is a polynomial 
algebra with generators
 $ c_1(EU(\ell, K)),\ldots, $\\
 $c_\ell(EU(\ell, K))$ truncated above
dim$(M) = \ell(K-\ell).$  By the previous proposition, $H^*(X,\BbC)$ surjects onto this 
algebra.  Letting $K$ go to infinity finishes the proof.  
\end{proof}

\begin{remark}
(i) A proof of  Theorem \ref{last theorem} for $H^*(B\calG,\BbC)$ that avoids most of the analysis
can be extracted from Lemma \ref{eq:IO} through
Proposition \ref{last prop}.

(ii) $E\calG(E)\to B\calG(E)$ is trivial as a $GL(\calH)$-bundle by Kuiper's theorem.
However, 
$E\calG(E)\to B\calG(E)$ 
is nontrivial as a $\calG(E)$-bundle, as it has nontrivial leading order characteristic classes.  
\end{remark}

We conclude with a result that complements Rochon's calculations of the homotopy groups of $\GG$ \cite{R}.

\begin{corollary}  In the setup of the Proposition \ref{last prop}, if $H^*_F(M,\BbC)$ is nontrivial, 
then $\el(F)$
is not a deformation retract of $\GG(F).$
\end{corollary}

\begin{proof}  Assume $\el$ is a deformation retract of $\GG$. Then every $\GG$-bundle
admits a reduction to a $\el$-bundle.
 Let $\calE\to B$ be a $\el$-bundle admitting a connection.  Lie$(\el)$ is the 
algebra of negative order $\pdo$s, so the connection and curvature forms have vanishing
leading order Chern classes.  For $B = {\rm Maps}(M,M)$ and $f = {\rm id}$, the
proof of Proposition \ref{last prop} gives an injection of $H_F(M,\BbC)$ into the subring of
$H^*({\rm Maps}(M,M),\BbC)$ generated by the leading order Chern classes.  This is a contradiction.
\end{proof} 

\bibliographystyle{amsplain}
\bibliography{Paper}

\end{document}